\documentclass[10pt]{amsart}
\usepackage{amsmath}
\usepackage{amsthm}
\usepackage{amssymb}
\usepackage{color}
\usepackage{amscd}
\usepackage{amsfonts}
\usepackage{graphicx}
\usepackage{epsfig}
\usepackage{amssymb,latexsym}
\usepackage{graphicx}

\newcommand{\supp}{\operatorname{supp}}
\newcommand{\adj}{\operatorname{adj}}
\newcommand{\dv}{\operatorname{div}}

\def\bu{\mathbf{u}}
\def\bv{\mathbf{v}}

\def\bx{\mathbf{x}}
\def\bz{\mathbf{z}}
\def\by{\mathbf{y}}

\def\bE{\mathbf{E}}

\def\bX{\mathbf{X}}

\def\bY{\mathbf{Y}}

\def\S{\mathbf{S}}
\def\R{\mathbb{R}}
\def\bR{\mathbf{R}}
\def\Z{\mathbf{Z}}
\def\C{\mathbf{C}}

\def\Md{\mathbf{M}^{2\times2}}
\def\Mt{\mathbf{M}^{3\times3}}
\def\Mdt{\mathbf{M}^{2\times3}}
\def\MN{\mathbf{M}^{N\times N}}
\def\MdN{\mathbf{M}^{2\times2N}}
\def\MtN{\mathbf{M}^{2\times N}}
\def\M{\mathbb{M}^{m\times N}}
\def\F{\mathbf F}
\def\G{\mathbf G}
\def\Y{\mathbf Y}

\def\bM{\mathbf M}

\def\bK{\mathbf K}

\def\bB{\mathbf B}
\def\id{\mathbf{1}}

\def\supp{\hbox{supp}}
\def\adj{\hbox{adj}}
\def\det{\hbox{det}}

\newtheorem{theorem}{Theorem}[section]
\newtheorem{corollary}[theorem]{Corollary}

\newtheorem{lemma}[theorem]{Lemma} 
\newtheorem{proposition}[theorem]{Proposition}

\begin{document}

\title{Inverse quasiconvexification}
\author{Pablo Pedregal}
\date{} 
\thanks{Supported by MICINN grant 
MTM2017-83740-P}
\begin{abstract}
In the context of the Calculus of Variations for non-convex, vector variational problems, the natural process of going from a function $\phi$ to its quasiconvexification $Q\phi$ is quite involved, and, most of the time, an impossible task. We propose to look at the reverse process, what might be called inverse quasiconvexification: start from a function $\phi_0$, and find functions $\phi$ for which $\phi_0=Q\phi$. In addition to establishing a few general principles, we show several explicit examples motivated by their application to inverse problems in conductivity. 
\end{abstract}

\maketitle

\section{Introduction}
A paradigmatic problem in the Calculus of Variations is that of finding the quasiconvexification $Q\phi(\F)$ of a certain integrand 
$$
\phi(\F):\M\to\R.
$$
The relevance of such a process is very well-established because the vector variational problem consisting in minimizing the integral
$$
\int_\Omega\phi(\nabla \bu(\bx))\,d\bx
$$
among all Lipschitz mappings 
$$
\bu(\bx):\Omega\subset\R^N\to\R^m
$$ 
with prescribed Dirichlet boundary datum, admits a relaxation in the similar form
$$
\int_\Omega Q\phi(\nabla \bu(\bx))\,d\bx.
$$
This sentence precisely means (\cite{DacorognaH}, \cite{PedregalI}) that the infima for both problems, the one with integrand $\phi$ and the one with integrand $Q\phi$, are equal over that class of mappings $\bu$; the problem with integrand $Q\phi$ admits minimizers (under additional conditions and over more specific spaces of functions that we overlook here), even though the one with $\phi$ might not; and there is a close connection between minimizing sequences for the first, and minimizers for the second. The formal definition of the relaxed integrand $Q\phi$ is (\cite{dacor})
$$
Q\phi(\F)=\inf_{\bu\in W^{1, \infty}_0(D, \R^m)}\frac1{|D|}\int_D \phi(\F+\nabla\bu(\by))\,d\by
$$
for a (arbitrary) Lipschitz domain $D\subset\R^N$ (this definition does not depend on the domain $D$ used). The passage $\phi\mapsto Q\phi$ is well beyond general techniques for the true vector situation ($m, N\ge2$), and only a few explicit examples are known under varying sets of conditions (check \cite{DacorognaH}). 

We would like to address what might be called the inverse quasiconvexification problem: 
\begin{quote}
Given a certain quasiconvex function $\phi_0$, describe or find functions $\phi$ such that $Q\phi=\phi_0$.
\end{quote}
There is always one such $\phi$, namely $\phi\equiv\phi_0$. Some times this is the only possibility, for instance when $\phi_0$ is strictly quasiconvex. So we would like to focus on cases where this is not the situation. Therefore there are two main issues to be addressed:
\begin{enumerate}
\item describe the structure of quasiconvex integrands $\phi_0$ for which there are more $\phi$'s than just $\phi_0$ itself with $Q\phi=\phi_0$; and
\item once one such $\phi_0$ is given, describe, if at all possible, all such $\phi$'s, or at least a non-trivial subset of them. 
\end{enumerate}
One fundamental issue is, no doubt, the last point: to discover explicit, non-trivial, interesting examples of, at least partial, inverse quasiconvexifications. We deal below with some such examples coming from other applied fields in Analysis. 

If $\phi_0=Q\phi$ so that $\phi_0\le\phi$, the coincidence set $\Z=\{\phi=\phi_0\}$ plays a central role. Off $\Z$, $\phi_0<\phi$ and gradient Young measures $\nu$ such that 
$$
\langle\phi, \nu\rangle=\phi_0(\langle\id, \nu\rangle)
$$
need to have their support precisely contained in $\Z$ (see Appendix \ref{ultimo} for more comments in this direction).
A general answer to the issue of inverse quasiconvexification which makes clear the role played by the coincidence set $\Z$ is the following. In these abstract terms is too general to be of some practical value, but it will be our guiding principle. 
\begin{proposition}\label{general}
Let 
$$
\phi_0(\F):\M\to\R
$$ 
be a quasiconvex function, and let $\Z\subset\M$ be closed. Let $\G\Y_\Z$ designate the set of all gradient Young measures supported in $\Z$. 
Define the set
$$
\tilde\Z=\{\F\in\M: \hbox{ there is }\nu_\F\in\G\Y_\Z, \hbox{ with barycenter }\F,\hbox{ and } \langle\nu_\F, \phi_0\rangle=\phi_0(\F)\}.
$$
For every function 
$$
\phi(\F):\M\to\R\cup\{+\infty\}
$$ 
such that
$$
\phi=\phi_0\hbox{ in }\Z\cup(\M\setminus\tilde\Z),\quad \phi\ge\phi_0\hbox{ in }\tilde\Z\setminus\Z,
$$
we have $Q\phi=\phi_0$.
\end{proposition}
The proof, which is easy, can be found in Section \ref{segunda}. 
Note that 
$$
\Z\subset\tilde\Z\subset Q\Z,
$$ 
if $Q\Z$ is the quasiconvexification of the set $\Z$ (see Appendix \ref{ultimo}). If there is no possibility of finding one such set $\Z$ with $\tilde\Z\setminus\Z\neq\emptyset$, then $\phi_0$ can only be the quasiconvexification of itself. Typically the set $\Z$ is sought as the coincidence set 
$$
\Z=\{\phi=\phi_0\}
$$ 
of a candidate $\phi$ for which $Q\phi=\phi_0$. Proposition \ref{general} provides then many other integrands with the same quasiconvexification. Note that there might be various feasible sets $\Z$, in the statement of Proposition \ref{general}, for the same underlying $\phi_0$. 

The truth is that Proposition \ref{general} is hard to apply in practice, as there is no a priori way to know if a given $\phi_0$ will accept a non-trivial $\tilde\Z$, or how many of these one could possibly find. Yet we will work with some explicit examples, the most important of which is the jacobian. Its statement requires the following notation. For an index $j$, $1\le j\le N$, put 
$$
\bM_j=\{(\alpha, \F)\in\R\times\MN: \alpha\F^{(j)}=\adj^{(j)}\F\},
$$
where $\F^{(j)}$ is the $j$-th row or column of $\F$, and $\adj^{(j)}\F$ is the $j$-th row or column, respectively, of the adjugate matrix. 
\begin{theorem}\label{jacobiano}
Suppose that 
$$
\phi(\F):\MN\to\R\cup\{+\infty\}
$$ 
is such that
$$
\phi(\F)=|\det\F|\hbox{ in }\bM_j,\quad \phi(\F)\ge|\det\F|\hbox{ off }\bM_j.
$$
Then $Q\phi(\F)=|\det\F|$. In particular, for
$$
\phi(\F):\MN\to\R, \quad \phi(\F)=|\adj^{(j)}\F|\,|\F^{(j)}|,
$$
we have $Q\phi(\F)=|\det\F|$. 
\end{theorem}
We will complete the proof of this result little by little, through successive versions of Proposition \ref{general}, and preliminary versions of Theorem \ref{jacobiano}. In addition, some extensions can be found in Section \ref{extensiones}. It is plausible that our results could be used in other explicit situations. Two final appendices have been included to cover some basic, known facts for the convenience of readers. In Appendix \ref{ultimo}, we have gathered statements and facts that are well-known to experts, and that are used without further comment throughout the paper. 

It is worthwhile to briefly describe the connection of some of these integrands to inverse problems in conductivity (\cite{astalapaivarinta}). This relationship will be much more deeply studied in a forthcoming contribution \cite{faustinopedregal}. For the sake of definiteness, let us consider the integrand
$$
\phi(\F):\Md\to\R,\quad \phi(\F)=|\F^{(1)}|\,|\F^{(2)}|, \F=\begin{pmatrix}\F^{(1)}\\\F^{(2)}\end{pmatrix},
$$
and its corresponding variational problem
\begin{equation}\label{ncvp}
\hbox{Minimize in }\bu:\quad \int_\Omega\phi(\nabla\bu(\bx))\,d\bx
\end{equation}
over a certain class of mappings $\bu$ having prescribed Dirichlet boundary data around $\partial\Omega$. This is a non-convex (and non-coercive), vector variational problem (\cite{PedregalI}). The Euler-Lagrange system for it is, at least formally,
\begin{equation}\label{elsistema}
\dv\left(\frac{|\nabla u_2(\bx)|}{|\nabla u_1(\bx)|}\nabla u_1(\bx)\right)=0,\quad
\dv\left(\frac{|\nabla u_1(\bx)|}{|\nabla u_2(\bx)|}\nabla u_2(\bx)\right)=0,
\end{equation}
if $\bu=(u_1, u_2)$. 
If we define the associated conductivity coefficient $\gamma(\bx)$ as
$$
\gamma(\bx)=\frac{|\nabla u_2(\bx)|}{|\nabla u_1(\bx)|},
$$
then 
$$
\dv(\gamma\nabla u_1)=0,\quad \dv(\frac1\gamma\nabla u_2)=0.
$$
These equations are exactly the ones for a couple of coherent measurements $(u_1, u_2)$ for the inverse conductivity problem. However, it is not clear under what circumstances problem \eqref{ncvp} would admit minimizers, in a way that it would be legitimate to ensure that there will be solutions for system \eqref{elsistema}. The relaxation of \eqref{ncvp} might play some role in understanding the situation. Note that 
this is a very particular case of Theorem \ref{jacobiano}. Its quasiconvexification is the jacobian function
$$
Q\phi(\F)=|\det\F|.
$$
There are many fundamental contributions on non-convex vector variational problems. The recent text \cite{rindler} is a very good place where most of the concepts and principal facts involved in varying frameworks are carefully and completely treated, and where those references can be found as well. 

\section{A basic principle}\label{segunda}
We start by proving Proposition \ref{general}. 
The inequality $\phi_0\le Q\phi$ is straightforward, given that $\phi_0$ is assumed to be quasiconvexity. Over the set 
$$
\Z\cup(\M\setminus\tilde\Z)
$$ 
there is nothing to show for in this set 
$$
\phi_0\le Q\phi\le\phi=\phi_0.
$$ 
Let $\F\in\tilde\Z\setminus\Z$. By definition of $\tilde\Z$ there is a certain gradient Young measure $\nu_\F$ with the claimed properties,  and we can put
$$
\phi_0(\F)\le Q\phi(\F)\le\langle\phi, \nu_\F\rangle=\langle\phi_0, \nu_\F\rangle=\phi_0(\F).
$$
Notice that we have used the fact that the quasiconvexification is the infimum over gradient Young measures, that $\nu_\F$ is supported in $\Z$, and that $\phi_0=\phi$ in $\Z$.

We will be trying to interpret the consequences of Proposition \ref{general}, and writing more transparent versions of it up to a point where specific examples can be found. 
A first statement in that direction follows. 
Recall that for a subset $\bK$ of matrices in $\M$, its quasiconvexification $Q\bK$ is the set of all possible first-moments of homogeneous gradient Young measures supported in $\bK$ (see Appendix \ref{ultimo}). Under no further restriction on the set $\bK$, various different definitions of its quasiconvex hull are possible (check for instance \cite{zhang}). But the one we adopt here is the best suited for our purposes. 

\begin{proposition}\label{curiosa}
Suppose we can write
$$
\Z=\cup_i\Z_i,\quad \M=\cup_i Q\Z_i,
$$
where the $\Z_i$'s are pairwise disjoint, and 
$$
\phi_0(\F):\M\to\R
$$ 
is quasiaffine over each $Q\Z_i$. For every 
$$
\phi(\F):\M\to\R
$$ 
such that 
$$
\phi=\phi_0\hbox{ in }\Z, \quad \phi\ge\phi_0\hbox{ off }\Z,
$$ 
we have
$$
Q\phi(\F)=\phi_0(\F).
$$
\end{proposition}
\begin{proof}
The proof is immediate just as the one of Proposition \ref{general}. If $\F\in Q\Z_i$, then there is at least one gradient Young measure $\nu$ such that
$$
\langle \id, \nu\rangle=\F, \quad \supp(\nu)\subset\Z_i.
$$
Then
$$
\phi_0(\F)\le Q\phi(\F)\le\langle \phi, \nu\rangle=\langle \phi_0, \nu\rangle=\phi_0(\F).
$$
The second inequality above holds because $\nu$ is a gradient Young measure; the first equality is correct because $\supp(\nu)\subset\Z$ where $\phi=\phi_0$; and the last one is due to the fact that $\phi_0$ is quasiaffine over $Q\Z_i$. 
\end{proof}
This situation can be applied to cases where $\phi_0$ is the supremum of quasiaffine functions
$$
\phi_0=\sup\{\phi_i\}
$$
and each $\phi_i$ is quasiaffine. $\phi_0$ is then quasiconvex (even polyconvex), and each set $\{\phi_0=\phi_i\}$ is quasiconvex by definition. If we aim at applying the preceding proposition in a non-trivial way, we need to find proper subsets $\Z_i$ of $\{\phi_0=\phi_i\}$ such that
\begin{equation}\label{cuasi}
Q\Z_i=\{\phi_0=\phi_i\}.
\end{equation}
This is again the inverse process to finding quasiconvexification of sets: instead of passing from $\Z_i$ to $Q\Z_i$, we would like to reverse the process and go from a known set $\tilde\Z_i(=\{\phi_0=\phi_i\})$ to a set $\Z_i$ such that $Q\Z_i=\tilde\Z_i$. The smaller the set $\Z_i$ is, the larger the set of functions $\phi$ whose quasiconvexification is $\phi_0$ will be. This is related to the difficult problem of finding the quasiconvex extreme points of a given set $\tilde\Z_i$ (\cite{zhang}, and also \cite{kruzik}). We do not pretend to get that far in this contribution, but will be contented with finding some explicit non-trivial situations. 

In practice, sets $\Z_i$ under condition \eqref{cuasi} are found in a direct way, by starting with a specific function $\phi$, in addition to $\phi_0$, the candidate to quasiconvexification, such that $\phi\ge\phi_0$ and writing 
$$
\Z=\cup_i\Z_i, \quad \Z_i=\{\phi=\phi_i\},\quad \phi_0=\sup_i\phi_i.
$$
The main part of the job is to show precisely that
$$
Q\{\phi=\phi_i\}=\{\phi_0=\phi_i\}.
$$

\section{One explicit example}\label{sec:4}
Consider the jacobian function
$$
\phi_0(\F):\Md\to\R,\quad \phi_0(\F)=|\det\F|.
$$
We would like to find one explicit family of functions $\phi$ such that $Q\phi=\phi_0$.
According to Proposition \ref{curiosa}, and bearing in mind that $\phi_0$ is quasiaffine over the sets of $2\times2$-matrices with a determinant of constant sign, we would need to find sets of matrices $\Z_+$, $\Z_-$ such that
\begin{equation}\label{ceroset}
Q\Z_\pm=\{\F\in\Md: \det\F\ge(\le)0\}.
\end{equation}
Recall that
$$
\det\F=-\F^{(1)}\cdot \bR\F^{(2)},
$$
if
$$
\bR=\begin{pmatrix}0&-1\\1&0\end{pmatrix}
$$
is the counterclockwise $\pi/2$-rotation in the plane. 
\begin{theorem}\label{primeroo}
Let 
$$
\phi(\F):\Md\to\R\cup\{+\infty\}
$$ 
be a function (no regularity assumed) such that
\begin{enumerate}
\item Coincidence set:
$$
\phi(\F)=\phi_0(\F),\quad \F\in\Z=\left\{\begin{pmatrix}\bx\\\alpha \bR\bx\end{pmatrix}: \alpha\in\R, \bx\in\R^2\right\};
$$
\item Off this coincidence set, we have
$$
\phi(\F)\ge\phi_0(\F),\quad \F\notin\Z.
$$
\end{enumerate}
Then $Q\phi(\F)=\phi_0(\F)$. Said differently, for every function $\phi$ such that
$$
\phi(\F)=|\det\F|,\quad \F\in\Z,
$$
we have
$$
Q(\max\{\phi(\F), |\det\F|\})=|\det\F|.
$$ 
\end{theorem}
Before proving this result, it is interesting to focus on the following particular example, which is a straightforward corollary of the previous theorem.

\begin{corollary}\label{ejemplo}
If 
$$
\phi(\F)=|\F^{(1)}|\,|\F^{(2)}|,\quad \F=\begin{pmatrix}\F^{(1)}\\
\F^{(2)}\end{pmatrix}\in\Md
$$
then
$$
Q\phi(\F)=|\det\F|.
$$
\end{corollary}
As readers may realize, our set $\Z$ in the statement of Theorem \ref{primeroo} is precisely given by the coincidence set
$\{\phi=\phi_0\}$ for this particular $\phi$ in Corollary \ref{ejemplo}. 

Though not of particular relevance for our purposes here, it is an interesting issue to know whether all matrices, or which among them, of the sets $\Z_\pm$ are quasiconvex extreme points (\cite{zhang}). Given that matrices in $\Z_+$ are $2$-quasiconformal matrices (\cite{astalafaraco}), there are special properties for gradient Young measures supported in $\Z_+$ (see also \cite{faraco}). 

\begin{proof}
According to Proposition \ref{curiosa}, all we need to check is that
$$
Q\Z_{\pm}=\{\F\in\Md: \det\F\ge(\le)0\}
$$
if
$$
\Z_\pm=\Z=\left\{\begin{pmatrix}\bx\\\alpha\bR\bx\end{pmatrix}: \alpha>(<)0, \bx\in\R^2\right\}.
$$

Note that the two matrices 
$$
\begin{pmatrix}\bx\\\alpha\bR\bx\end{pmatrix},\quad \begin{pmatrix}\by\\\beta\bR\by\end{pmatrix}
$$
are rank-one connected if
$$
(\bx-\by)\cdot(\alpha\bx-\beta\by)=0.
$$

Suppose  first that $\F$ has positive determinant. 
We will concentrate on showing that  two matrices $\F_0, \F_1\in \Z_+$, and a parameter $t\in[0, 1]$ can be found, such that
$$
\F=t\F_1+(1-t)\F_0,\quad \F_1-\F_0, \hbox{ rank-one}.
$$
The computations that follow are based on similar calculations, for instance in \cite{pedregalSS}, in a slightly different framework. 

We already know that  
$$
\F_i=\begin{pmatrix}\bx_i\\\alpha_i\bR\bx_i\end{pmatrix},\quad i=1, 0,
$$
for some positive $\alpha_i$ and vectors $\bx_i$. The condition on the difference $\F_1-\F_0$ being a rank-one matrix translates, as already remarked, into
\begin{equation}\label{productonulo}
(\bx_1-\bx_0)\cdot(\alpha_1\bx_1-\alpha_0\bx_0)=0;
\end{equation}
finally, we should have
$$
\F^{(1)}=t\bx_1+(1-t)\bx_0,\quad \F^{(2)}=t\alpha_1\bR\bx_1+(1-t)\alpha_0\bR\bx_0.
$$
From these two vector equations, one can easily find that
\begin{gather}
\bx_1=\frac1t\frac1{\alpha_0-\alpha_1}(\alpha_0\F^{(1)}+\bR\F^{(2)}),\nonumber\\
\bx_0=-\frac1{1-t}\frac1{\alpha_0-\alpha_1}(\alpha_1\F^{(1)}+\bR\F^{(2)}).\nonumber
\end{gather}
If we replace these expressions in \eqref{productonulo}, and rearrange terms, we arrive at the quadratic equation in $t$
\begin{align}
\det\,\F \,t^2-&\frac1{\alpha_0-\alpha_1}(\alpha_1\alpha_0|\F^{(1)}|^2-|\F^{(2)}|^2+(\alpha_0-\alpha_1)\det\,\F)\, t\nonumber\\
+&\frac1{(\alpha_0-\alpha_1)^2}(\alpha_0^2\alpha_1|\F^{(1)}|^2+\alpha_1|\F^{(2)}|^2-2\alpha_0\alpha_1\det\,\F)=0.\label{cuadratico}
\end{align}
The value of this quadratic function for $t=0$ and $t=1$ turns out to be, respectively,
$$
\frac{\alpha_1}{(\alpha_0-\alpha_1)^2}|\alpha_0\F^{(1)}+\bR\F^{(2)}|^2,\quad
\frac{\alpha_0}{(\alpha_0-\alpha_1)^2}|\alpha_1\F^{(1)}+\bR\F^{(2)}|^2.
$$
Under the condition $\det\,\F>0$, there are roots for $t$ in $(0, 1)$, provided that the discriminant is non-negative, and the vertex of the parabola belongs to $(0, 1)$. It is elementary, again after some algebraic manipulations, that these conditions amount to having
\begin{equation}\label{ultimacondicion}
2\sqrt{\alpha_1\alpha_0}\sqrt{|\F^{(1)}|^2|\F^{(2)}|^2-\det\,\F\,^2}\le
(\alpha_1+\alpha_0)\det\,\F-\alpha_1\alpha_0|\F^{(1)}|^2-|\F^{(2)}|^2,
\end{equation}
for some positive values $\alpha_i$, $i=1, 0$. If we examine the function of two variables
$$
f(\alpha_1, \alpha_0)=\frac1{\sqrt{\alpha_1\alpha_0}}[(\alpha_1+\alpha_0)\det\,\F-\alpha_1\alpha_0|\F^{(1)}|^2-|\F^{(2)}|^2],
$$
we realize that along the hyperbole $\alpha_1\alpha_0=1$, $f$ grows indefinitely (recall that $\det\,\F>0$), and eventually it becomes larger than any positive value, in particular, bigger than
$$
2\sqrt{|\F^{(1)}|^2|\F^{(2)}|^2-\det\,\F\,^2}.
$$
In this way \eqref{ultimacondicion} is fulfilled for some positive values for $\alpha_1$ and $\alpha_0$, and the proof of this step is finished. 

If $\det\,\F<0$, it is readily checked that the same above calculations lead to the result $Q\phi(\F)=-\det\,\F$ because there is a minus sign in front of every occurrence of the determinant, with negative values for $\alpha_1$ and $\alpha_0$. 
\end{proof}
Once these computations have been checked out, one realizes that the general $N$-th version of the result in this corollary (including that of Corollary \ref{tres} below) can be shown, and generalized, by taking into account the Hadamard inequality
\begin{equation}\label{deter}
|\det\F|\le\Pi_i|\F^{(i)}|,\quad \F=\begin{pmatrix}\F^{(i)}\end{pmatrix},
\end{equation}
and equality holds (the coincidence set) precisely when the rows (or columns) $\F^{(i)}$ are orthogonal . The rank-one convex envelope of the right-hand side in \eqref{deter} yields back the jacobian on the left. 

\section{A general principle}
We would like to push the ideas of our basic principle Proposition \ref{curiosa} to build some other examples. In particular, for a quasiconvex function 
$$
\phi_0(\F):\M\to\R\cup\{+\infty\},
$$
of the form
\begin{equation}\label{inverso}
\phi_0(\F)=\max\{\psi(\F), \max\{\phi_\lambda(\F): \lambda\in\Lambda\}\},
\end{equation}
where $\psi$ is quasiconvex and each $\phi_\lambda$ is quasiaffine, we would like to describe all possible functions $\phi$ such that $Q\phi=\phi_0$. We explicitly separate the function $\psi$ because it will play a different role compared to the quasiaffine terms $\phi_\lambda$. In order to avoid undesirable situations, we make explicit assumptions that could, otherwise, be taken tacitly for granted, namely, 
\begin{enumerate}
\item the sets $\{\phi_0=\phi_\lambda\}$ are non-empty; 
\item the set $\M\setminus\{\phi_0=\psi\}$ is bounded; and
\item the function $\psi$ is strictly quasiconvex.
\end{enumerate}

\begin{theorem}\label{principal}
Under the assumptions just indicated, a function
$$
\phi(\F):\M\to\R\cup\{+\infty\}
$$
is such that $Q\phi=\phi_0$ given in \eqref{inverso} if and only if there are sets 
$$
 \bM_\lambda\subset\{\phi_0=\phi_\lambda\},
$$
with
\begin{equation}\label{clave}
 Q\bM_\lambda=\{\phi_0=\phi_\lambda\}
\end{equation}
for all $\lambda\in\Lambda$, and if
$$
\bM_0=\{\phi_0=\psi\},
$$
then we have
\begin{gather}
\phi=\phi_0\hbox{ on }\bM_0\cup\left(\cup_\lambda\bM_\lambda\right),\nonumber\\
\phi\ge\phi_0\hbox{ off }\bM_0\cup\left(\cup_\lambda\bM_\lambda\right).\nonumber
\end{gather}
\end{theorem}
\begin{proof}
The proof follows along the lines of the preceding discussion. Note that
$$
\psi\le\phi_0\le\phi,\quad \phi_\lambda\le\phi_0\le\phi,
$$
and
$$
\bM_0=\{\phi=\phi_0=\psi\}.
$$
Because $\psi$, $\phi_0$, and $\phi_\lambda$ all are quasiconvex, we always have
$$
\psi\le\phi_0\le Q\phi, \quad \phi_\lambda\le\phi_0\le Q\phi.
$$
If there are sets $\bM_0$, $\bM_\lambda$ with the indicated properties, then for a matrix $\F\in \bM_0$, we would have
$$
\phi(\F)=\phi_0(\F)=\psi(\F)\le Q\phi(\F)\le\phi(\F),
$$
and so $Q\phi(\F)=\phi_0(\F)$. If, on the other hand, $\F\in Q\bM_\lambda$ and so there is some (homogeneous) gradient Young measure $\nu$ with
$$
\F=\langle\nu, \id\rangle,\quad \supp(\nu)\subset\bM_\lambda\subset\{\phi=\phi_\lambda\},
$$
then
$$
\phi_0(\F)\le Q\phi(\F)\le\langle\nu, \phi\rangle=\langle\nu, \phi_\lambda\rangle.
$$
But since $\phi_\lambda$ is quasiaffine,
$$
\langle\nu, \phi_\lambda\rangle=\phi_\lambda(\F)=\phi_0(\F)
$$
because of \eqref{clave}. Hence $Q\phi(\F)=\phi_0(\F)$ as well.

Conversely, suppose there is a function $\phi$ with $\phi_0=Q\phi$. The strict quasiconvexity assumed on $\psi$ implies that $\phi=\psi$ whenever $\psi=\phi_0$, and hence the coincidence set
$$
\Z=\{\phi_0=\phi\}
$$
is non-empty. Put
$$
\bM_0=\Z\cap\{\phi_0=\psi\},\quad \bM_\lambda=\Z\cap\{\phi_0=\phi_\lambda\}.
$$
Clearly $\bM_\lambda\subset\{\phi_0=\phi_\lambda\}$. Since $\phi_\lambda$ is quasiaffine, if $\F\in Q\bM_\lambda$, 
$$
 \phi_\lambda(\F)=\langle\nu, \phi_\lambda\rangle=\langle\nu, \phi_0\rangle
 $$
 for some gradient Young measure $\nu$ supported in $\bM_\lambda$ where $\phi_0=\phi_\lambda$. If $\phi_0=Q\phi$ a quasiconvex function, then
 $$
\phi_\lambda(\F)\le  \phi_0(\F)\le \langle\nu, \phi_0\rangle.
 $$
 Altogether we see that $\phi_0(\F)=\phi_\lambda(\F)$, and 
 $$
 Q\bM_\lambda\subset\{\phi_0=\phi_\lambda\}.
 $$
If, on the other hand, $\F$ is such that 
$$
Q\phi(\F)=\phi_0(\F)=\phi_\lambda(\F),
$$ 
then there is a gradient Young measure $\nu$ with support in the coincidence set $\Z$ and barycenter $\F$ such that, because of the quasiaffinity of $\phi_\lambda$,
$$
\langle\nu, \phi_\lambda\rangle=\phi_\lambda(\F)=Q\phi(\F)=\langle\nu, \phi\rangle.
$$
On the one hand $\phi-\phi_\lambda\ge0$, but on the other its integral against the probability measure $\nu$ vanishes. We can therefore conclude that 
$$
\supp(\nu)\subset\Z\cap\{\phi=\phi_\lambda\},
$$
i. e. $\F\in Q\bM_\lambda$. The other statements are straightforward if we take into account, once again, that $\phi=Q\phi=\phi_0$ in $\Z$ and $\phi>\phi_0$ off $\Z$. 
\end{proof}

As we see from this theorem, every quasiconvex function $\phi_0$ of the form \eqref{inverso} is always a quasiconvexification. Having interesting examples of integrands $\phi$ having such quasiconvexification $Q\phi=\phi_0$ depends on our ability to find generating sets $\bM_\lambda$. 

\section{Some examples}
We treat in this section examples of the form
\begin{equation}\label{tres}
\phi(\F)=|\F^{(1)}\times\F^{(2)}|\,|\F^{(3)}|,\quad \F=\begin{pmatrix}\F^{(1)}\\
\F^{(2)}\\\F^{(3)}\end{pmatrix}\in\Mt,
\end{equation}
where $\bu\times\bv$ is the vector product in $\R^3$, for which we can find its quasiconvexification. As a matter of fact, it is as cheap to treat the general $N$-dimensional situation. We would like to address the question of finding as many functions 
$$
\phi(\F):\MN\to\R
$$
as possible so that $Q\phi=\phi_0$ with $\phi_0(\F)=|\det\F|$. 
We can find initially at least $2N$ such different integrands all having the same quasiconvexification $\phi_0$. 
\begin{theorem}\label{adjunto}
Let
$$
\phi(\F):\MN\to\R,\quad \phi(\F)=|\adj^{(j)}\F|\,|\F^{(j)}|,
$$
where $\adj^{(j)}\F$ is the $N$-vector corresponding to the $j$-th column or row of the adjugate matrix of $\F$, and $\F^{(j)}$ is the $j$-th column- or row of $\F$, respectively, for some $j\in\{1, 2, \dots, N\}$. 
Then
$$
Q\phi(\F)=|\det\F|,\quad \F\in\MN.
$$
\end{theorem}
\begin{proof}
The case $N=2$ has been treated in Corollary \ref{ejemplo}. We assume hence $N\ge3$.
It is clear that it suffices to treat one of those $2N$ possible cases. For definiteness, put
$$
\phi(\F):\MN\to\R,\quad \phi(\F)=|\adj^{(N)}\F|\,|\F^{(N)}|,
$$
where $\adj^{(N)}\F$ is the $N$-th, row-wise adjugate, $N$-vector of matrix $\F$, and $\F^{(N)}$ is the $N$-th row of $\F$.

It is elementary to realize that
$$
\phi_0(\F)=|\det\F|=\max\{\det\F, -\det\F\}
$$
with both $\pm\det\F$ quasiaffine, is of the form \eqref{inverso} (with no $\psi$). According to Theorem \ref{principal}, we need to identify two sets of matrices
$$
\bM_+\subset\{\F: \det\F>0\},\quad \bM_-\subset\{\F: \det\F<0\}
$$
such that
$$
Q\bM_+=\{\F: \det\F\ge0\},\quad Q\bM_-=\{\F: \det\F\le0\},
$$
and check that
$$
\phi=\phi_0\hbox{ in }\bM_+\cup\bM_-,\quad 
\phi\ge\phi_0\hbox{ off }\bM_+\cup\bM_-.
$$
We therefore examined first the set
$$
\bM_+=\{\phi(\F)=\det\F\}.
$$
It is straightforward to find, given that 
$$
\det\F=\adj^{(N)}\F\cdot\F^{(N)}
$$
(the same is true for all $2N$ possible cases), that 
$$
\bM_+=\{\F\in\MN: \alpha\F^{(N)}=\adj^{(N)}\F, \alpha>0\}.
$$
We can conclude through Theorem \ref{principal} as soon as we can prove that
$$
Q\bM_+=\{\F: \det\F\ge0\},
$$
since arguments for the negative part are symmetric. 

Assume a matrix $\F$ is such that
$$
\F=t\F_1+(1-t)\F_0,\quad \F_1-\F_0,\hbox{ rank-one}, \F_i\in\bM_+, i=1, 0, t\in[0, 1].
$$
Because all adjugate functions are rank-one affine, we know
$$
\adj^{(N)}\F=t\,\adj^{(N)}\F_1+(1-t)\,\adj^{(N)}\F_0,
$$
in addition to
$$
\F^{(N)}=t\F_1^{(N)}+(1-t)\F_0^{(N)}.
$$
Since each $\F_i\in\bM_+$, $i=1, 0$, we have altogether
$$
\adj^{(N)}\F=t{\alpha_1}\F_1^{(N)}+{(1-t)}{\alpha_0}\F_0^{(N)},\quad \F^{(N)}=t\F_1^{(N)}+(1-t)\F_0^{(N)}.
$$
Let us put, for the sake of notational simplicity $\bx_i=\F^{(N)}_i$, $i=1, 0$, so that
\begin{equation}\label{linsis}
\adj^{(N)}\F= t{\alpha_1}\bx_1+{(1-t)}{\alpha_0}\bx_0,\quad \F^{(N)}=t\bx_1+(1-t)\bx_0.
\end{equation}
We can solve for vectors $\bx_i$ in this system to find
\begin{gather}
\bx_0=\frac1{(1-t)(\alpha_1-\alpha_0)}(\alpha_1\F^{(N)}-\adj^{(N)}\F),\nonumber\\
\bx_1=\frac1{t(\alpha_1-\alpha_0)}(\adj^{(N)}\F-\alpha_0\F^{(N)}).\nonumber
\end{gather}
Since $\F_1-\F_0$ is rank-one, in particular, its determinant vanishes, and bearing in mind that $\F_i\in\bM_+$ and $\bx_i=\F^{(N)}_i$, we need to enforce
$$
0=(\alpha_1\bx_1-\alpha_0\bx_0)\cdot(\bx_1-\bx_0).
$$
If we substitute the formulas for $\bx_i$ in terms of $\F$, $t$ and $\alpha_i$, we conclude 
\begin{equation}\label{conjcero}
0=(\adj^{(N)}\F-(t\alpha_1+(1-t)\alpha_0)\F^{(N)})\cdot((t\alpha_0+(1-t)\alpha_1)\adj^{(N)}\F-\alpha_1\alpha_0\F^{(N)}).
\end{equation}
Regard $t$, $\alpha_1$, and $\alpha_0$ as fixed, and consider the polynomial $P(\F)\equiv P_{t, \alpha_1, \alpha_0}(\F)$ of degree $2N-2$ in $\F$ given by
$$
P_{t, \alpha_1, \alpha_0}(\F)=(\adj^{(N)}\F-(t\alpha_1+(1-t)\alpha_0)\F^{(N)})\cdot((t\alpha_0+(1-t)\alpha_1)\adj^{(N)}\F-\alpha_1\alpha_0\F^{(N)}).
$$
Its leading part is, given that $N\ge3$, is 
$$
P_0(\F)\equiv P_{t, \alpha_1, \alpha_0, 0}(\F)=(t\alpha_0+(1-t)\alpha_1)|\adj^{(N)}\F|^2.
$$
\eqref{conjcero} implies that
\begin{equation}\label{inclusion}
\{\F: P_{t, \alpha_1, \alpha_0}(\F)=0\}\subset Q\bM_+
\end{equation}
for each such triplet $(t, \alpha_1, \alpha_0)$. In addition, 
two main points, that are elementary to check, are:
\begin{enumerate}
\item $P_0(\F)\ge0$ for all $\F$, and it is not identically zero on the rank-one cone;
\item $P(\F)$ is rank-one convex because written in the form
\begin{align}
P(\F)=&P_0(\F)-(\alpha_1\alpha_0+(t\alpha_1+(1-t)\alpha_0)(t\alpha_0+(1-t)\alpha_1)\det\F\nonumber\\
&+\alpha_1\alpha_0(t\alpha_1+(1-t)\alpha_0)|\F^{(N)}|^2,\nonumber
\end{align}
we see that  it is, in fact, polyconvex. 
\end{enumerate}
Lemma \ref{primero} in Appendix \ref{appuno} permits us to ensure, for each fixed triplet $(t, \alpha_1, \alpha_0)$, that the rank-one envelope of the set
$\{P(\F)=0\}$ in \eqref{conjcero} is the sub-level set $\{P(\F)\le0\}$. Therefore, 
if one can show that for given $\F$ with positive determinant, one can always find values of $t\in[0, 1]$, and positive $\alpha_i$, $i=1, 0$, so that $P(\F)\le0$, then our result will be proved. Indeed, if this is so we would have
\begin{equation}\label{primerainclusion}
\{\F: \det\F>0\}\subset\cup_{t\in[0, 1], \alpha_i>0}\{\F: P_{t, \alpha_1, \alpha_0}(\F)\le0\},
\end{equation}
and then
\begin{align}
\{\F: \det\F>0\}&\subset\cup_{t\in[0, 1], \alpha_i>0}\{\F: P_{t, \alpha_1, \alpha_0}(\F)\le0\}\nonumber\\
&=\cup_{t\in[0, 1], \alpha_i>0}R\{\F: P_{t, \alpha_1, \alpha_0}(\F)=0\}\nonumber\\
&\subset \cup_{t\in[0, 1], \alpha_i>0}Q\{\F: P_{t, \alpha_1, \alpha_0}(\F)=0\}\nonumber\\
&\subset Q\bM_+\nonumber\\
&\subset\{\F: \det\F\ge0\}.\nonumber
\end{align}
Note how we have used here \eqref{inclusion}, and the facts that $\det$ is quasiaffine, and the rank-one convex envelope $R$ of a set of matrices is always a subset of the quasiconvexification $Q$ of the same set. 

There are various ways of checking \eqref{primerainclusion} as we have a lot of freedom. Assume $\F$ is given with positive determinant, and take $t=1/2$. Then
\begin{equation}\label{expresion}
P(\F)=\frac{\alpha_1+\alpha_0}2\left(|\adj^{(N)}\F|^2-\frac{\alpha_1+\alpha_0}2\det\F+\alpha_1\alpha_0|\F^{(N)}|^2\right)-\alpha_1\alpha_0\det\F.
\end{equation}
Given the form of the expression within parenthesis in \eqref{expresion},
if we further demand that
\begin{equation}\label{raices}
\alpha_1+\alpha_0=4\frac{|\adj^{(N)}\F|^2}{\det\F},\quad \alpha_1\alpha_0=\frac{|\adj^{(N)}\F|^2}{|\F^{(N)}|^2}
\end{equation}
the term within parenthesis in \eqref{expresion} vanishes, and then 
$$
P(\F)=-\frac{|\adj^{(N)}\F|^2}{|\F^{(N)}|^2}\det\F<0.
$$
Note that if $\det\F$ is positive, $\F^{(N)}$ cannot vanish. The values of $\alpha_1$ and $\alpha_0$ in \eqref{raices} are the roots of the quadratic polynomial
$$
\alpha^2-4\frac{|\adj^{(N)}\F|^2}{\det\F}\alpha+\frac{|\adj^{(N)}\F|^2}{|\F^{(N)}|^2}=0.
$$
Again, since $\det\F=\adj^{(N)}\F\cdot\F^{(N)}$, it is elementary to check that this polynomial admits two positive real roots $\alpha_1$ and $\alpha_0$. 

This full discussion, and the corresponding symmetric argument for matrices with negative determinant, show that 
$$
Q\bM_+=\{\F: \det\F\ge0\},\quad Q\bM_-=\{\F: \det\F\le0\},
$$
and our result is proved.
\end{proof}
A direct corollary of Theorem \ref{principal}, right after Theorem \ref{adjunto}, allows to find more functions $\psi$ for which $Q\psi=\phi_0$, once we have one. 
\begin{corollary}\label{gen}
Let $\phi_0(\F)$ be given in \eqref{inverso}, and let 
$$
\phi(\F):\M\to\R
$$ 
be such that $\phi_0=Q\phi$. Put $\Z=\{\phi=\phi_0\}$. 
If a further function $\psi(\F):\M\to\R$ is such that 
$$
\psi\ge\phi, \quad\Z=\{\psi=\phi_0\},
$$
then $Q\psi=Q\phi=\phi_0$. 
\end{corollary}
\begin{proof}
The inequality $Q\psi\ge\phi_0$ is straightforward because 
$$
Q\phi=\phi_0\le\phi\le\psi
$$ 
and $\phi_0$, being a quasiconvex hull, is quasiconvex. On the other hand, Theorem \ref{principal} implies the existence of sets $\bM_\lambda$ and $\bM_0$ with the properties indicated in the statement of the theorem. It is clear, because of our hypotheses 
$$
\Z=\{\psi=\phi_0\},\quad \psi\ge\phi,
$$
that the same family of sets $\bM_\lambda$, $\bM_0$ enable the application of Theorem \ref{principal} for $\psi$ as well. Hence $Q\psi=\phi_0$.
\end{proof}

\section{Some extensions}\label{extensiones}
There are various ways to extend the previous examples. A first possibility is to consider
$$
\phi(\F)=|\F^{(1)}|\,|\F^{(2)}|\,|\F^{(3)}|,\quad \F=\begin{pmatrix}\F^{(1)}\\
\F^{(2)}\\\F^{(3)}\end{pmatrix}\in\Mt.
$$
Even though it is true that
$$
\phi(\F)\ge|\F^{(1)}\times\F^{(2)}|\,|\F^{(3)}|\ge\phi_0(\F),\quad \phi_0(\F)=|\det\F|,
$$
Corollary \ref{gen} cannot be used directly to conclude anything because the coincidence set $\{\phi=\phi_0\}$ is strictly smaller than 
$$
\{|\F^{(1)}\times\F^{(2)}|\,|\F^{(3)}|=\phi_0\},
$$ 
and further work is required to show that nevertheless we still have $Q\phi=\phi_0$. 

Other interesting extensions motivated by the use of these variational principles in inverse problems (\cite{faustinopedregal}) are the following
\begin{gather}
\psi_N(\F)=\sum_{i=1}^N\phi(\F_i)=\sum_{i=1}^N|\F^{(1)}_i|\,|\F^{(2)}_i|,\nonumber\\
\phi_N(\F)=\sqrt{\sum_{i=1}^N|\F^{(1)}_i|^2}\,\sqrt{\sum_{i=1}^N|\F^{(2)}_i|^2},,\nonumber\\
\F=\begin{pmatrix}\F_1&\F_2&\dots&\F_N\end{pmatrix}=
\begin{pmatrix}\F^{(1)}_1&\F^{(1)}_2&\dots&\F^{(1)}_N\\
\F^{(2)}_1&\F^{(2)}_2&\dots&\F^{(2)}_N\end{pmatrix}\in\MdN,\nonumber
\end{gather}
for a positive integer $N$. There are corresponding versions for $3\times3$ matrices. It is easy to argue that
$$
Q\psi_N(\F)=\sum_{i=1}^N|\det\F_i|,
$$
however, the identity
$$
Q\phi_N(\F)=\left|\sum_{i=1}^N\det\F_i\right|
$$
asks for more insight. 

The most interesting example in this section is however the following. For
$$
\F=\begin{pmatrix}\F_1&\F_2&\dots&\F_N\end{pmatrix}
=\begin{pmatrix}\F^{(1)}\\\F^{(2)}\end{pmatrix}=
\begin{pmatrix}F^{(1)}_1&F^{(1)}_2&\dots&F^{(1)}_N\\
F^{(2)}_1&F^{(2)}_2&\dots&F^{(2)}_N\end{pmatrix}\in\MtN,
$$
put
$$
\phi(\F)=|\F^{(1)}|\,|\F^{(2)}|.
$$
Depending on the particular value of $N$, we would like to select a collection $M_{ij}$, $(i, j)\in\Lambda$ of $2\times2$-minors of $\F$ such that $Q\phi(\F)=\phi_0(\F)$, where
$$
\phi_0(\F)=\sqrt{\sum_{(i, j)\in\Lambda}M_{ij}(\F)^2}\quad\hbox{or}\quad
\phi_0(\F)=\left|\sum_{(i, j)\in\Lambda}M_{ij}(\F)\right|.
$$
Note that $\phi_0(\F)$ is a polyconvex function in both situations.
The case $N=2$ has already been explored earlier. For this value of $N=2$, both forms of $\phi_0$ collapse to the same underlying function. 
We are here especially interested in the values $N=3$, and $N=2N$, an even number. In these two cases, we will take, respectively,
$$
\phi_0(\F)=|\F^{(1)}\times\F^{(2)}|=\sqrt{M_{12}(\F)^2+M_{13}(\F)^2+M_{23}(\F)^2},\quad \phi_0(\F)=\left|\sum_{i=1}^N\det\F_i\right|,
$$
where 
$$
\F=\begin{pmatrix}\F_1&\F_2&\dots&\F_N\end{pmatrix}\in\MdN,
$$
and each $\F_i$ is a $2\times2$-matrix. 
Note that we always have
$$
\Lambda\subset\{(i, j): 1\le i<j\le N\}.
$$
\begin{theorem}
If 
\begin{gather}
\phi_N(\F)=|\F^{(1)}|\,|\F^{(2)}|=\sqrt{\sum_{i=1}^N|\F^{(1)}_i|^2}\,\sqrt{\sum_{i=1}^N|\F^{(2)}_i|^2},\nonumber\\
\F=\begin{pmatrix}\F_1&\F_2&\dots&\F_N\end{pmatrix}=\begin{pmatrix}\F^{(1)}\\\F^{(2)}\end{pmatrix}=
\begin{pmatrix}\F^{(1)}_1&\F^{(1)}_2&\dots&\F^{(1)}_N\\
\F^{(2)}_1&\F^{(2)}_2&\dots&\F^{(2)}_N\end{pmatrix}\in\MdN,\nonumber
\end{gather}
we have
$$
Q\phi_N(\F)=\left|\sum_{i=1}^N\det\F_i\right|.
$$
\end{theorem}
\begin{proof}
Let $\bR$ be, as ususal, the $\pi/2$-counterclockwise rotation in the plane. By a natural abuse of language, we will also put
$$
\bR:\R^{2N}\to\R^{2N},\quad \bR\bx=\bR(\bx_1, \bx_2, \dots, \bx_N)\mapsto
(\bR\bx_1, \bR\bx_2, \dots, \bR\bx_N),
$$
for $\bx_i\in\R^2$, $\bx=(\bx_1, \bx_2, \dots, \bx_N)\in\R^{2N}$. Note that $\bR^2=-\id$, minus the identity mapping, and
$$
-\F^{(1)}\cdot\bR\F^{(2)}=\sum_{i=1}^N\det\F_i.
$$
Formally, computations are similar to the ones in the proof of Corollary \ref{ejemplo}. Indeed, the coincidence set
$$
\Z=\{\phi_N=\phi_0\},\quad \phi_0(\F)=\left|\sum_{i=1}^N\det\F_i\right|
$$
can be written again i the form
$$
\Z=\left\{\begin{pmatrix}\bx\\\alpha \bR\bx\end{pmatrix}: \alpha\in\R, \bx\in\R^{2N}\right\}.
$$
We have a similar result to that in the proof of Corollary \ref{ejemplo} in the sense
$$
Q\Z_\pm=\{\F\in\MdN: \sum_{i=1}^N\det\F_i>(<)0\}.
$$
Calculations in the proof of of Corollary \ref{ejemplo} are formally the same, though the quadratic equation \eqref{cuadratico} becomes, after rearranging terms, 
\begin{align}
\alpha_1\alpha_0((1-t)\alpha_0+t\alpha_1)|\F^{(1)}|^2&+((1-t)\alpha_1+t\alpha_0)|\F^{(2)}|^2\nonumber\\
&+((\alpha_0-\alpha_1)^2t-2\alpha_0\alpha_1)\sum_{i=1}^N\det\F_i=0.\nonumber
\end{align}
Let $P_2(\F)$ be the second-degree polynomial in the entries of $\F$, for given $t\in[0, 1]$, $\alpha_1>0$, $\alpha_0>0$, on the left-hand side of this equation. It is immediate to check that Lemma \ref{dosene} below can be applied, and so we conclude that the quasiconvexification of the zero set $\{P_2=0\}$ is the sub-level set $\{P_2\le0\}$. As we argued earlier in the proof of Theorem \ref{adjunto}, it suffices to check that for arbitrary $\F\in\MdN$ with $\sum_i\det\F_i>0$, it is always possible to find $t\in[0, 1]$, and positive $\alpha_1$, $\alpha_0$ so that $P_2(\F)\le0$. This is similar to the parallel calculations in the proof of Corollary \ref{ejemplo}. 
\end{proof}
For the case of $\Mdt$ one has the following.
\begin{theorem}\label{caso23}
Put
$$
\phi(\F)=|\F^{(1)}|\,|\F^{(2)}|,\quad \F=\begin{pmatrix}\F^{(1)}\\\F^{(2)}\end{pmatrix}\in\Mdt, \F^{(i)}\in\R^3, i=1, 2.
$$
Then
$$
Q\phi(\F)=\phi_0(\F)=|\F^{(1)}\times\F^{(2)}|
$$
where $\times$ indicates vector product in $\R^3$.
\end{theorem}
\begin{proof}
It is elementary to have
$\phi(\F)\ge\phi_0(\F)$, and because $\phi_0$ is polyconvex, $Q\phi(\F)\ge\phi_0(\F)$. The coincidence set $\Z=\{\phi=\phi_0\}$ is given by
$$
\Z=\{\F\in\Mdt: \F^{(1)}\cdot\F^{(2)}=0\}.
$$
The following is an elementary fact. 
\begin{lemma}\label{basicoo}
Let $\bx, \by$ be two independent, non-orthogonal vectors in $\R^2$, and put
$$
\lambda=-\frac{\bx\cdot\by}{|\bx\cdot\by|}\in\{-1, 1\}.
$$
A non-vanishing vector $\bz\in\R^2$ can be found in such a way that if
$$
\bx_\pm=\bx\pm\bz,\quad \by_\pm=\by\pm\lambda\bz,
$$
then
\begin{enumerate}
\item orthogonality:
$$
\bx_+\cdot\by_+=\bx_-\cdot\by_-=0;
$$
\item parallelism: $\bx_+-\by_+$ is proportional to $\bx_--\by_-$ (and to $\bz$);
\item representation:
$$
\bx=\frac12\bx_++\frac12\bx_-,\quad \by=\frac12\by_++\frac12\by_-.
$$
\end{enumerate}
\end{lemma}
\begin{proof}
If vector $\bz$ is chosen in the intersection of the two circles
$$
(\bz-\bx)\cdot(\bz-\lambda\by)=0,\quad (\bz+\bx)\cdot(\bz+\lambda\by)=0,
$$
then it is elementary to check all the claimed conditions. The choice of $\lambda$ ensures, because the origen belongs to the interior of both circles, that they have a non-empty intersection. Once $\bz$ is chosen in this way, it is straightforward to check the three requirements in the statement. Note that $\lambda=1/\lambda$. 
\end{proof}
Suppose now, going back to the proof of  our theorem, that $\F\in\Mdt$ is an arbitrary matrix. If the rows $\F^{(1)}$ and $\F^{(2)}$ are orthogonal, $\F\in\Z$. If not, and 
assuming by density that the two rows of $\F$ are independent, it is always possible to work in a plane $\pi$ containing $\F^{(1)}$ and $\F^{(2)}$. If we apply Lemma \ref{basicoo} in the plane $\pi$ and to the two vectors
$$
\bx=\F^{(1)},\quad \by=\F^{(2)},
$$
we can find matrices $\F_1$ (with rows $\bx_+$ and $\bx_-$), $\F_0$ (with rows $\by_+$ and $\by_-$),  belonging to $\Z$ with the additional properties that $\F^{(j)}_i\in\pi$ for $j=1, 2$, $i=1, 0$, and such that $\F_1-\F_0$ is rank-one and 
$$
\F=\frac12\F_1+\frac12\F_0.
$$
Because all rows involved belong to the same plane $\pi$, it is also immediately checked that the function
$$
t\mapsto|(t\F^{(1)}_1+(1-t)\F^{(1)}_0)\times(t\F^{(2)}_1+(1-t)\F^{(2)}_0)|
$$
is affine in $t$ given that it never vanishes. Indeed, the two vectors
$$
t\F^{(1)}_1+(1-t)\F^{(1)}_0,\quad t\F^{(2)}_1+(1-t)\F^{(2)}_0
$$
can never be collinear if one relies on their form given through Lemma \ref{basicoo}. This is elementary. 

All of these facts imply, because of the arbitrariness of $\F$, that, with the notation of Proposition \ref{general}, the set $\tilde\Z$ is all of $\Mdt$. 
The conclusion is then a direct consequence of Proposition \ref{general}.
\end{proof}
If we put together this result with Theorem \ref{adjunto}, we are able to conclude
\begin{corollary}\label{tres}
Put 
$$
\phi(\F)=|\F^{(1)}|\,|\F^{(2)}|\,|\F^{(3)}|,\quad \F=\begin{pmatrix}\F^{(1)}\\\F^{(2)}\\\F^{(3)}\end{pmatrix}\in\Mt, \F^{(i)}\in\R^3, i=1, 2, 3.
$$
Then
$$
Q\phi(\F)=|\det\F|.
$$
\end{corollary}
\begin{proof}
For the proof, notice that because there is no interaction between the two submatrices
$$
\begin{pmatrix}\F^{(1)}\\\F^{(2)}\end{pmatrix},\quad \F^{(3)}
$$
of $\F$ in $\phi$, we will have, because quasiconvexification works in the same way for inhomogeneous integrands, 
$$
Q\phi(\F)=Q\left(Q(|\F^{(1)}|\,|\F^{(2)}|)\,|\F^{(3)}|\right)=
Q(|\F^{(1)}\times\F^{(2)}|\,|\F^{(3)}|)=|\det\F|,
$$
by Theorems \ref{caso23} and \ref{adjunto}. 
\end{proof}

\section{Appendix. Auxiliary results}\label{appuno}
The results in this section, or slight variations of them, were proved in \cite{boussaidpedregal}, and even before in \cite{boussaid} and \cite{boussaid2}. 
\begin{lemma}\label{primero}
Let $P(\F):\MN\to\R$ be a polynomial of degree $2N-2$, $N\ge3$, with leading part $P_0(\F):\MN\to\R$ so that $P_0(\F)$ is homogeneous of degree $2N-2$. Suppose there is a rank-one matrix $\F_1$ such that 
$$
P_0(\F_1)>0,\quad P_0(-\F_1)>0.
$$
Then the rank-one convexification $R\Z_0$ of the zero set
$$
\Z_0=\{\F\in\MN: P(\F)=0\}
$$
contains the sub-level set
$$
\Z_-=\{\F\in\MN: P(\F)\le0\}.
$$
If, in addition, the polynomial $P(\F)$ is quasiconvex then $Q\Z_0=\Z_-$. Moreover, if there is another rank-one matrix $\F_2$ such that
$$
P_0(\F_2)<0,\quad P_0(\F_2)<0,
$$
then $Q\Z_0=\MN$.
\end{lemma}
There is a similar version for $\MtN$ matrices that we include here for the sake of completeness. This particular version is exactly the one that can be found in \cite{boussaidpedregal}.
\begin{lemma}\label{dosene}
Let $P(\F):\MtN\to\R$ be a polynomial of degree $2N-2$, $N\ge3$, with leading part $P_0(\F):\MtN\to\R$ so that $P_0(\F)$ is homogeneous of degree $2N-2$. Let $\land$ any cone in $\MtN$.
Suppose there is a matrix $\F_1\in\land$ such that 
$$
P_0(\F_1)>0,\quad P_0(-\F_1)>0.
$$
Then the $\land$-convexification $\land\Z_0$ of the zero set
$$
\Z_0=\{\F\in\MtN: P(\F)=0\}
$$
contains the sub-level set
$$
\Z_-=\{\F\in\MtN: P(\F)\le0\}.
$$
If, in addition, the polynomial $P(\F)$ is $\land$-convex then $\land\Z_0=\Z_-$. Moreover, if there is another matrix $\F_2\in\land$ such that
$$
P_0(\F_2)<0,\quad P_0(\F_2)<0,
$$
then $\land\Z_0=\MtN$.
\end{lemma}

The main tool in proving this kind of facts is the following lemma whose proof we briefly include here for the convenience of readers.
\begin{lemma}
Let $\land$ be any cone in a certain Euclidean space $\R^q$. Let $P(\bX)$ be a real function defined on $\R^q$ such that there are positive reals $d_1<d_2<\dots<d_n$ and homogeneous of degree $d_i$ functions $P_i(\bX)$ with
$$
P(\bX)=\sum_iP_i(\bX).
$$
Suppose that there exists $\bE\in\land$ such that
$$
P_n(\bE)>0, \quad P_n(-\bE)>0.
$$
If $\F\in\R^q$ is such that $P(\F)\le\alpha$, then there are two vectors $\bB, \C\in\R^q$, and $s\in[0, 1]$ such that
$$
\F=s\bB+(1-s)\C,\quad P(\bB)=P(\C)=\alpha,\quad \bB-\C\in\land.
$$
\end{lemma}
\begin{proof}
Suppose  that $P(\F)\le\alpha$.
 Let 
$$
 \bB(t)=\F+t\bE,\quad \C_t(\lambda)=\F-{\lambda\over1-\lambda}t\bE
 $$ 
 for $\lambda\in[0, 1)$. Then for every $t\in \R$ and each $\lambda\in[0,1)$ we have
$$
\F=\lambda \bB(t)+(1-\lambda)\C_t(\lambda), \hbox{ and }
(\bB(t)-\C_t(\lambda))\in\land.
$$
Consider the function $t\mapsto P(\bB(t))$. For $t=0$,
$P(\bB(0))=P(\F)\le\alpha$. On the other hand, for $t$ large we make use of the homogeneity 
\begin{eqnarray*}
P(\bB(t))&=&P_1(\bB(t))+P_2(\bB(t))+\dots+P_n(\bB(t))\\&=&P_1(\bB+t\bE)+P_2(\bB+t\bE)+\dots+P_n(\bB+t\bE)\\&=&t^{d_1}P_1({1\over
t}\bB+\bE)+t^{d_2}P_2({1\over
t}\bB+\bE)+\dots+t^{d_n}P_n({1\over
t}\bB+\bE)\\&=&t^{d_n}\left[t^{(d_1-d_n)}P_1({1\over
t}\bB+\bE)+t^{(d_2-d_n)}P_2({1\over t}\bB+\bE)+\dots+P_n({1\over
t}\bB+\bE)\right].
\end{eqnarray*}
Then
$$
\lim_{t\to+\infty}P(\bB(t))=\lim_{t\to+\infty}t^{d_n}P_n(\bE)=+\infty.
$$
By  continuity,  there exists $t_0>0$ such that $P(\bB(t_0))=\alpha.$
For this value $t_0$, we focus on $\C_{t_0}(\lambda)$, and consider the
function $\lambda\in[0,1)\mapsto h(\C_{t_0}(\lambda)).$ For
$\lambda =0$, $P(\C_{t_0}(0))=P(\F)<\alpha$, and arguing as above we
have
$$
\lim_{\lambda\to1^-}P(\C_{t_0}(\lambda))=\lim_{\lambda\to1^-}t_0^{d_n}({\lambda\over1-\lambda})^{d_n}P_n(-\bE)=+\infty.
$$
By continuity again, there exists a real $\lambda_0\in]0,1[$ such that
$P(\C_{t_0}(\lambda_0))=\alpha.$
\end{proof}

\section{Appendix}\label{ultimo}
Most of the basic concepts involved in this contribution are well-known to specialists in the area of non-convex vector variational problems. We simply gather here various statements to facilitate the understanding of the scope of our results, and provide some standard references for interested readers. 

Young measures have turned out to be an accepted way to deal with weak convergence and non-linear integral functionals (\cite{Ball2}). When these families of probability measures are generated by sequences of gradients, they are called gradient Young measures (\cite{PedregalI}). It is important to stress this point, as it is of paramount importance to bear in mind the fact that having gradients of functions is always a requirement. Results are much easier to understand if we neglect this gradient condition, as we fall back to usual notions of convexity (\cite{DacorognaH}). Though it is also important to pay attention to spaces where these generating sequences of gradients belong to, we will simply consider sequences of gradients  of uniformly bounded Lipschitz functions. We can put $\G\bY(\M)$ for the full set of homogeneous (not depending on the point $\bx$ in the domain $\Omega\subset\R^N$ considered) gradient Young measures that can be generated by a sequence of gradients of uniformly bounded Lipschitz fields with $m$ components. 
\begin{itemize}
\item Let 
$$
\phi(\F):\M\to\R\cup\{+\infty\}
$$
be an integrand. The function
\begin{equation}\label{cuasiconvexificacion}
Q\phi(\F)=\inf\{\langle\phi, \nu\rangle: \nu\in\G\bY(\M), \langle\id, \nu\rangle=\F\}
\end{equation}
is called the quasiconvexification of $\phi$. If $Q\phi$ turns out to yield back $\phi$, we say that $\phi$ is quasiconvex. The remarkable fact that place these convex hulls in an important role is the coincidence of the two infima
$$
\inf\{\int_\Omega \phi(\nabla\bu(\by))\,d\by: \bu=\bu_0\hbox{ on }\partial\Omega\}
$$
and
$$
\inf\{\int_\Omega Q\phi(\nabla\bu(\by))\,d\by: \bu=\bu_0\hbox{ on }\partial\Omega\},
$$
under appropriate classes of competing fields $\bu$ that we do not bother to specify here. The result is valid even for inhomogeneous integrands $\phi(\by, \F)$. 
\item It is a fact that
$$
Q\phi=\sup\{\psi: \psi\le\phi, \phi, \hbox{quasiconvex}\},
$$
and that the quasiconvexification of a function is a quasiconvex function on its own. 
\item There is a special subclass of $\G\bY(\M)$, the so-called laminates $\L(\M)$ (\cite{PedregalI}), which, in fact, is the collection of those that are used in practice in computations. They follow a natural, recursive law that is quite helpful in many ways (\cite{DacorognaE}). 
\item The elements of $\G\bY(\M)$ realizing the infimum in \eqref{cuasiconvexificacion} enjoy special properties. The most important is the localization of its support: for one such $\nu$ we will have
$$
\supp(\nu)\subset\{\phi=Q\phi\}.
$$
\item This same quasiconvexification concept can also be applied to sets $\S\subset\M$ of matrices. Though there are several different but equivalent ways to define these convex hulls of sets, one possibility is to define
$$
Q\S=\{\langle\id, \nu\rangle: \nu\in\G\bY(\M), \supp(\nu)\subset\S\}.
$$
The same applies to the rank-one convexification of $\S$, namely
$$
R\S=\{\langle\id, \nu\rangle: \nu\in\L(\M), \supp(\nu)\subset\S\}.
$$
\item Quasiconvex functions that are not convex are not easy to find. The main such source is the class of polyconvex functions. They are built upon the so-called quasiaffine functions which are those  $\phi(\F)$ for which both $\phi$ and $-\phi$ are quasiconvex. These are known to be exactly the linear functions of the full set of minors (of any size) of $\F$. Polyconvex functions are then convex (in the usual sense) functions of all such minors. Finally, another important collection of functions is the class of rank-one convex functions which are those that are convex, at least, along rank-one convex directions. Quasiconvex functions are always rank-one convex. There is a deep parallelism between gradient Young measures and quasiconvex functions, on the one hand, and laminates and rank-one convex functions on the other. It is established through Jensen's inequality (\cite{kinder}). 
\end{itemize}

\end{document}